\documentclass[12pt]{amsart}

\usepackage{amsmath,amssymb,amsfonts,euscript,hyperref}
\hypersetup{
pdftitle={Research Paper},
pdfauthor={Abbas Nasrollah Nejad},
pdfsubject={Torsion-free Aluffi Algebra },
pdfcreator={Abbas Nasrollah Nejad},
colorlinks,
linkcolor=blue,
citecolor=magenta,
anchorcolor=red,
bookmarksopen,
urlcolor=red,
filecolor=red,
pdfpagetransition={Wipe}
}

\theoremstyle{plain}
\newtheorem{Theorem}{Theorem}[section]
\newtheorem{Corollary}[Theorem]{Corollary}

\newtheorem{Proposition}[Theorem]{Proposition}
\newtheorem{Conjecure}[Theorem]{Conjecture}
\newtheorem{Question}[Theorem]{Question}
\theoremstyle{definition}
\newtheorem{Definition}[Theorem]{Definition}
\newtheorem{Example}[Theorem]{Example}
\theoremstyle{Remark}
\newtheorem{Remark}{Remark}

\newcommand{\lar}{\longrightarrow}
\newcommand{\surjects}{\twoheadrightarrow}

\newcommand{\overbar}[1]{\mkern 1.5mu\overline{\mkern-1.5mu#1\mkern-1.5mu}\mkern 1.5mu}

\def\FF{{\bf F}}

\def\fm{{\mathfrak m}}

\def\fp{{\mathfrak p}}

\def\ffr{{\mathfrak f}}

\def\gr#1#2{{\rm gr}\, _{#1}(#2)}
\def\gr{{\rm gr}\,}
\def\hht{{\rm ht}\,}

\def\spec{{\rm spec}\,}

\def\ker{{\rm ker}\,}

\def\syz{\mbox{\rm Syz}}

\def\spec#1{{\rm Spec}\, (#1)}
\def\proj#1{{\rm Proj}\, (#1)}

\def\Rees{{\mathcal R}}

\def\Rees{\mathcal{R}}
\def\syz{\mathcal{Z}}

\def\pp{{\mathbb P}}

\begin{document}

\title[Torsion-free Aluffi Algebras]{Torsion-free Aluffi Algebras}
\author[A. Nasrollah Nejad, Z. Shahidi, R. Zaare-Nahandi ]{Abbas Nasrollah Nejad, Zahra Shahidi, Rashid Zaare-Nahandi }
\address{department of mathematics institute for advanced studies in basic sciences (IASBS) p.o.box 45195-1159  zanjan, iran.
}
\email{abbasnn@iasbs.ac.ir, z.shahidi@iasbs.ac.ir, rashidzn@iasbs.ac.ir}

\begin{abstract}
 A pair of ideals $J\subseteq I\subseteq R$ has been called Aluffi torsion-free if the Aluffi algebra of $I/J$ is isomorphic with the corresponding Rees algebra. We  give necessary and sufficient conditions for the Aluffi torsion-free property  in  terms of the first syzygy module of the  form ideal $J^*$ in the associated graded ring of $I$. For two pairs of ideals $J_1,J_2\subseteq I$ such that $J_1-J_2\in I^2$, we prove that  if one pair is Aluffi torsion-free the other one is so if and only if the first syzygy modules of $J_1$ and $J_2$ have the same form ideals.  We introduce the notion of strongly Aluffi torsion-free ideals and present some results on these ideals.
 \end{abstract}
\maketitle
\section*{Introduction}
P. Aluffi in (\cite{aluffi}) to describe characteristic cycles of a hypersurface parallel to well known conormal cycle in intersection theory  introduces an intermediate graded  algebra between the symmetric algebra of an ideal and the corresponding Rees algebra. The first author and A. Simis~(\cite{AA}) called such an  algebra, the \textit{Aluffi algebra}. Given a Notherian ring  $R$ and
ideals $J\subset I$ of $R$, the Aluffi algebra of $I/J$ is defined by
$$\mathcal{A}_{_{R/J}}(I/J):=\mathrm{Sym}_{R/J}(I/J)
\otimes_{\mathrm{Sym}_R(I)}\Rees_R(I).$$
The Aluffi algebra is squeezed as $\mathrm{Sym}_{R/J}(I/J)\surjects \mathcal{A}_{_{R/J}}(I/J)\surjects\Rees_{R/J}(I/J)$ and 
moreover it is a residue ring of the ambient Rees algebra $\Rees_R(I)$. The kernel of the right hand surjection so called the module of  \textit{Valabrega-Valla} as defined in (\cite{Wolmbook1}),  is the torsion of the Aluffi algebra. Thus the Rees algebra of $I/J$ is the Aluffi algebra modulo its torsion  provided that  $I$ has a regular element modulo $J$.

It is reasonable to ask when the the surjection $\mathcal{A}_{_{R/J}}(I/J)\surjects\Rees_{R/J}(I/J)$ is  an isomorphism?  Geometrically, this question is important  form  two points of view. More precisely, the blowup of $X=\spec{R/J}$ along the closed subscheme $Y$ defined by the ideal $I/J$ is equal to  $\proj{\mathcal{A}_{R/J}(I/J)}$.  Hence to find the 
equations of the blowup of $X$ along $Y$, we just need to find the equations of the blowup of ambient space $R$ along $X$. For the other one, let $x\in X=\spec{R/J}$ be a point then the tangent cone of $X$ at  $x$ is the cone $\spec{\gr_{\fm_x}(\mathcal{O}_{X,x})}$, where $\fm_x$ is the maximal ideal of the local ring  $\mathcal{O}_{X,x}\simeq R_{\fp_x}/JR_{\fp_x}$. So that we may assume  $\mathcal{O}_{X,x}$ is the quotient of a regular local ring $(R,\fm)$ with respect of an ideal $J$. Then the associated graded ring  $\gr_{\fm/J}(R/J)$ of $\fm/J$ is isomorphic to $\gr_{\fm}(R)/J^*$ where $J^*$ is the form ideal. The problem of determining elements $f_1,\ldots,f_t$ such that their initial forms $f_1^*,\ldots,f_t^*$ generate $J^*$ is an essential problem  in resolution of singularities. Also the torsion-free Aluffi algebras are crucial in intersection theory of regular and linear embedding (\cite{fulton}, \cite{keel}). The outline of this paper is as the follow.

In section 1, we give necessary and sufficient conditions for  torsion-free Aluffi algebra, involving the standard base (in the sense of Hironaka (\cite{H.H})) and  the first syzygy module of the form ideal in the associated graded ring.
 
Let $J\subseteq I$ be ideals in the ring $R$.  We say that the pair $J\subseteq I$ is \textit{Aluffi torsion-free} if $J\cap I^n=JI^{n-1}$ for all $n\geq 1$.
In section 2, we study the behavior of the Aluffi torsion-free property with respect to contraction and extension. We prove that the sum of two Aluffi torsion-free ideals  is Aluffi torsion-free if and only if one of them modulo the other is Aluffi torsion-free. As the main result of this section, we prove that if $J_1, J_2\subseteq I$ such that $J_1\equiv J_2$ modulo $I^2$ and $J_1\subseteq I$ is Aluffi torsion-free then $J_2\subseteq I$ is Aluffi torsion-free if and only if the first syzygy modules of $J_1$ and $J_2$ have the same form ideals in the associated graded module  $\gr_I(R^m)$ where $m$ is the number of generators of $J_1$ and $J_2$ (Theorem \ref{I-adic}). In sequel, we introduce the notion of strongly Aluffi torsion-free ideals. A pair $J=(f_1,\ldots,f_t)\subseteq I\subseteq R$ is called \textit{strongly  Aluffi torsion-free}  if $J_i=(f_1,\ldots,f_i)$ is Aluffi torsion-free for $i=1,\ldots, t$. We give an example of Aluffi torsion-free pair of  ideals which is not strongly Aluffi torsion-free. In the case that, $J\subseteq I$ is Aluffi torsion-free, we give a criterion for strongly Aluffi torsion-freeness. We close the section with this result:  let  $J_1,J_2\subseteq I$ be ideals in the ring $R$ such that the extension of $J_2$ and $I$ in the ring $R/J_1$ is Aluffi torsion-free. If  there exists a minimal generating set $f_1,\ldots,f_t$ of $J_1$ such that $J_1\subseteq I$ is strongly Aluffi torsion-free and extension of the sequence $f_1,\ldots,f_t$ in $R/J_2$ is regular then $J_2\subseteq I$ is Aluffi torsion-free. 

In section 3,  we focus on the case that $J$ is an ideal in the polynomial ring $R=k[x_0,\ldots,x_n]$ over a field $k$ of characteristic zero and the ideal $I$ stands for the Jacobian ideal of $J$ which describe the singular subscheme  of  $\spec {R/J}$. We prove that if $J$ is the ideal of a  monomial curve with some special parametrization or $J$ is the  square-free Veronese ideal of degree $r$, then $J\subseteq I$ is Aluffi torsion-free.  We close the paper with a question related  to Aluffi torsion-freeness of free line arrangements.

\section{The Aluffi algebra and its torsion}
Throughout this section $R$ will be a Notherian ring. Let  $ J \subseteq I\subseteq R$ be ideals. There are two important algebras related to these data. The first one is the Symmetric algebra  $\mathrm{Sym}_{R}(I)$  and the second one is the Rees algebra,
$\mathcal{R}_{R}(I)=\bigoplus_{n \geq 0} I^{n}t \subset R[t] $.
It is well-known that there is a natural surjective $ R-$algebra homomorphism 
$ \mathrm{Sym}_{R}(I) \twoheadrightarrow \mathcal{R}_{R}(I)$. By functorial  property of Symmetric algebra there is an other surjection $ \mathrm{Sym}_{R}(I)\twoheadrightarrow \mathrm{Sym}_{R/J}(I/J)$. The (emmbeded) Aluffi algebra is defined by 
$$ \mathcal{A}_{R/J}(I/J)= \mathrm{Sym}_{R/J}(I/J)\otimes_{\mathrm{Sym}_{R}(I)} \mathcal{R}_{R}(I).$$
By \cite[Lemma 1.2]{AA}, there are $R$-algebra isomorphisms
$$ \mathcal{A}_{R/J}(I/J)\simeq \mathcal{R}_{R}(I)/(J,\tilde{J}) \mathcal{R}_{R}(I) \simeq\bigoplus _{n \geq 0} I^{n}/JI^{n-1}, $$
where $J$ is in degree zero and $\tilde{J}$ is in degree 1. The Rees algebra of $I/J$ is 
\[ \mathcal{R}_{R/J}(I/J)\simeq \bigoplus_{n\geq 0} I^n/J\cap I^{n}. \]
Then there is a surjective $R$-algebra homomorphism 
$  \mathcal{A}_{R/J}(I/J) \twoheadrightarrow \mathcal{R}_{R/J}(I/J).$
The kernel of the above surjection is the homogeneous ideal  $$ \mathcal{W}_{ J\subset I}:=\bigoplus_{n \geq 2} J \cap I^{n}/JI^{n-1},$$ 
which is called the \textit{ module of Valabrega-Valla}. If $J$ has a regular element modulo $I$, the Valabrega-Valla 's module is the $R/J$-torsion of the Aluffi algebra \cite[Proposition 2.5]{AA}.
The module of Valabrega-Valla has close relation to the theory of standard base (\cite{H.H}). To make a further development on this relation, we recall some facts about the filtered rings and  modules. 

A filtration on the ring $R$ is a decreasing sequence of ideals $\{\FF_n R\}_{n\geq 0}$ satisfying $(\FF_nR)(\FF_mR)\subseteq \FF_{n+m}R$ for all $n,m\geq 0$. The pair $(R,\FF_nR)$ is called a \textit{filtered ring}.  For an ideal $I$ of a ring $R$ there is the $I$-adic filtration $\FF_n R=I^n$. A morphism of filtered rings $\varphi: (R,\FF_n R)\lar (S, \FF_n S) $ is a homomorphism  of rings $\varphi: R\lar S$ such that  $\varphi(\FF_n R)\subseteq \FF_n S$ for all $n\geq 0$. 

Let $(R,\FF_n R)$ be a filtered ring and $M$ a $R$-module. A filtration on the module $M$ is a decreasing sequence $\{\FF_n M\}_{n\geq 0 }$ of submodules of $M$ such that $(\FF_nR)(\FF_nM)\subseteq \FF_{n+m} M$ for all $m,n\geq 0$. The pair $(M,\FF_n M)$ is called a \textit{filtered $(R,\FF_nR)$-module}. A morphism of filtered $(R,\FF_n R)$-modules $\varphi: (M,\FF_n M)\lar (N,\FF_n N)$ is a $R$-module homomorphism $\varphi: M\lar N$ such that $\varphi(\FF_n M)\subseteq\FF_n N $ for all $n\geq 0$. This implies $\varphi(\FF_n M)\subseteq \varphi(M)\cap \FF_n N$. The morphism $\varphi$ is called \textit{strict} if $\varphi(\FF_n M)= \varphi(M)\cap \FF_n N$. If $M$ is a $R$-module then $(M,\FF_n M)$ with $\FF_n M:=(\FF_n R)M$ is a filtered $(R,\FF_n R)$-module. A sequence of filtered $(R,\FF_n R)$-modules is called exact if the sequence of underlying $R$-modules is exact. It is called strict if all morphisms are strict. 
\begin{Remark}\label{Inj-Sur-Sum}
Let $(M,\FF_n M)$ be a filtered $(R,\FF_n R)$ module.
\begin{enumerate}
	\item[{\rm (a)}]  Let  $\varphi: L\lar M$ an  injective  homomorphism  of $R$-modules. For all $n\geq 0$, put $\FF_n L:=\varphi^{-1}(\FF_n M)$. This makes $(L,\FF_n L)$ into a filtered module and $\varphi: (L,\FF_n L)\lar (M,\FF_n M) $ is a strict morphism.
	\item[{\rm (b)}] Let $\varphi: M\lar N$ be a surjective homomorphism  of $R$-modules. For all $n\geq 0$ put $\FF_n N:=\varphi(\FF_n M)$. This makes  $(N,\FF_n N)$ into a filtered module and $\varphi: (M,\FF_n M)\lar (N,\FF_n N) $ is a strict morphism. 
\end{enumerate}
 \end{Remark}
The \textit{associated graded ring}  of a filtration $ (R,\FF_n R)$ is $\gr(R)=\bigoplus_{n\geq 0} \FF_n R/\FF_{n+1} R$. We denote by  $\gr_I(R)$ for the $I$-adic filtration. If $(M,\FF_n M)$ is a filtered module, its \textit{associated graded module} $\gr(M)=\bigoplus_{n\geq 0} \FF_n M/\FF_{n+1} M$ is the graded $\gr(R)$-module. In the case of $I$-adic filtration we write $\gr_I(M)$. It is clear that $\gr(-)$ is a functor form the category of filtered modules to the category of graded modules. 
\begin{Proposition}[\cite{Griaud}, I Proposition 2.1]\label{Girud1}
Let $(R,\FF_nR)$ be a filtered ring and $$(L,\FF_n L)\stackrel{\varphi}\lar (M,\FF_n M) \stackrel{\phi}\lar (N,\FF_n N),$$  a strict exact sequence of  filtered $(R,\FF_nR)$-modules. Then the induced sequence $\gr(L)\stackrel{\gr(\varphi)}\lar\gr(M)\stackrel{\gr(\phi)}\lar \gr(N)$ is an exact sequence of $\gr(R)$-modules. 
\end{Proposition}
If $m\in M$ we denote by $\nu_{\FF}(m)$ the largest integer $n$ such that $m\in \FF_n M$. If such $n$ dose not exist we say $\nu_{\FF}(m)=\infty$ and if $\nu_{\FF}(m)< \infty$, we denote by $m^*$ the residue class of $m$ in $\FF_{\nu_{\FF}(m)}M/\FF_{\nu_{\FF}(m)+1}M$, which is called the \textit{initial form} of $m$. If $\nu_{\FF}(m)=\infty$, then we set $m^*=0$. For $m_1,m_2\in M$ if $m_1^*+m_2^*\neq 0$, then $m_1^*+m_2^*=(m_1+m_2)^*$. If the filtration $\FF_n M$ is multiplicative then $\gr(M)$ is a ring and if $m_1^*m_2^*\neq 0$ then $(m_1m_2)^*=m_1^*m_2^*$.  
 
Let  $R$ be a ring and $J\subseteq I$ ideals of $R$. Given an element  $f\in R$, we denote by $\nu=\nu_I(f)$ the number $\nu_{\FF}(f)$ with  $\FF_nR=I^n$.  We denote by $J^*$ the homogeneous ideal of $\gr_I(R)$ generated by the initial forms of the elements of $J$.   A set of generators $ \{f_{1} ,\ldots, f_{t}\}$ of $ J$ is called $ I$-\textit{standard base } if $ J^{*}=(f_{1}^{*},\ldots,f_{t}^{*})$. When  $ R$ is local, then an $ I$-standard base of $ J$ is a generating set \cite[Lemma 6]{H.H}. 

The following remark give necessary and sufficient conditions for the surjection $\mathcal{A}_{R/J}(I/J) \twoheadrightarrow \mathcal{R}_{R/J}(I/J)$ to be an isomorphism.      
\begin{Remark}\label{VV-Aluffi}\rm
Let $J\subseteq I\subseteq R$ be ideals of the local ring $R$. By \cite[Theorem 1.1]{VaVa} the following are equivalent.
\begin{enumerate}
\item[{\rm (a)}]$ \mathcal{A}_{R/J}(I/J)\simeq \mathcal{R}_{R/J}(I/J)$. 
\item[{\rm (b)}]$ J\cap I^{n}=JI^{n-1}$ for any $ n\geq 1$.
\item[{\rm (c)}]$ I^{n+1} \cap JI^{n-1}=JI^{n}$ for any $ n \geq 1 $.
\item[{\rm (d)}] There exists a minimal set of generators $f_1,\ldots, f_t$ of $J$ such that $\{f_1,\ldots, f_t\}$ is a $I$-standard base of $J$ and  $ \nu_I(f_{i})=1$ for $i=1,\ldots, t$.

\end{enumerate} 

\end{Remark}

Let now $J=(f_1,\ldots,f_t)$ and $\nu_I(f_i)=1$. Consider the exact sequence 
\begin{equation}\label{syz-sequence}
0\lar \mathcal{Z}\stackrel{i}\hookrightarrow R^t\stackrel{\mathfrak{f}}\lar R\stackrel{\pi}\twoheadrightarrow  R/J\lar 0,
\end{equation}
where $\mathfrak{f}(a_1,\ldots,a_t)=\sum_{i=1}^{t}a_if_i$ and $\mathcal{Z}=\mathrm{Syz}(J)$ is the first syzygy module of $J$. By Remark (\ref{Inj-Sur-Sum}), we consider the following filtrations 
\[\FF_n R^t=\bigoplus_{i=1}^t I^{n-1}\quad,\quad \FF_n \syz=\FF_nR^t\cap \syz\quad,\quad \FF_n R/J=(I/J)^n,\]
which  make $i, \mathfrak{f}$ and $\pi$ the morphisms of filtered $(R,\FF_nR=I^n)$-modules. 
 Note that $i$ and $\pi$ are strict. By Proposition (\ref{Girud1}), we get the corresponding complex of graded modules
\begin{equation}\label{gr-I}
\nonumber 0\lar \gr_I(\syz)\stackrel{\gr(i)}\lar \gr_I(R^t) \stackrel{\gr(\mathfrak{f})}\lar \gr_I(R) \stackrel{\gr(\pi)}\lar \gr_{I/J}(R/J)\lar 0.\quad\quad  (*)
\end{equation}
Note that $\gr_I(R^t)=\bigoplus_{i=1}^t \gr_I(R)(-1)$. The map $\gr(\mathfrak{f})$ is defined by $e_i\mapsto f_i^*$, the map $\gr(i)$ is inclusion  and $\gr(\pi)$ is surjective. We have 
\[\ker(\gr(\pi))=\bigoplus_{n\geq 0}(I^{n+1}+J\cap I^n)/I^{n+1} =J^*\quad ,\quad \ker(\gr(\mathfrak{f}))=\mathrm{Syz}(J^*). \]
Given an element $(a_1,\ldots,a_t)\in \syz$, if $(a_1,\ldots,a_t)\neq(0,\ldots,0)$ there exist $m\geq 0$ such that $\nu_{\FF_n\syz}(a_1,\ldots,a_t)=m$, this means that $\nu_I(a_j)\geq m-1$ for every $i$ and there exists $j\in \{1,\ldots, t\}$
 such that $\nu_I(a_j)=m-1$.   
Hence  $ \psi : \mathcal{Z} \longrightarrow \gr_{I}(\mathcal{Z})$ is the canonical map which associates to every element of $\mathcal{Z} $ its initial form  in $ \gr_{I}(\mathcal{Z})$, that is,  $\psi(a_1,\ldots,a_t)=\overbar{(a_1,\ldots,a_t)}\in \FF_m\syz/\FF_{m+1}\syz$. On the other hand, since the sequence $(*)$ is a complex hence  there is a canonical embedding $\varphi: \gr_I(\syz)\hookrightarrow \mathrm{Syz}(J^*)$ which  sends every element  $\overline{(a_1,\ldots,a_t)}\in \syz \cap \FF_mR^t/\syz\cap \FF_{m+1}R^t$ to $(\overline{a_1},\ldots, \overline{a_t})$ where  $\overline{a_i}$ is the residue class of $a_i$ in $I^{m-1}/I^m$. Therefore, we get a map 
$$\varphi\circ\psi: \syz\lar \mathrm{Syz}(J^*)\quad,\quad \varphi\circ\psi((a_1,\ldots,a_t))=(\overline{a_1},\ldots, \overline{a_t}).$$
Note that $\overline{a_i}=a_i^*$ if $\nu_I(a_i)+1=\min_j\{\nu_I(a_j)+1\}$ and  $ \overbar{a_{i}}=0 $ if $ \nu(a_{i})+1 > \min_j\{\nu(a_{j})+1\}$. 
The following theorem relate the torsion of the Aluffi Algebra to the first syzygy module  of  the form ideal $J^*\subseteq \gr_I(R)$. 
\begin{Theorem}\label{Torsion-Syzygy}
Let $ J=(f_1,\ldots, f_t) \subseteq I$ be ideals in the local ring $R$. The following are equivalent. 
	\begin{enumerate}
		\item[{\rm (a)}]$ \mathcal{A}_{R/J}(I/J)\simeq \mathcal{R}_{R/J}(I/J)$.
		\item[{\rm (b)}] The complex $(*)$ is exact. 
		\item[{\rm (c)}] There exist a homogeneous system of generators of $\mathrm{Syz}(J^*)$, whose elements can be lifted to elements of $\syz$ via $\varphi\circ\psi$. 
		
	\end{enumerate}
\end{Theorem}
\begin{proof}
First note that the Aluffi algebra is torsion-free if and only if the map $\mathfrak{f}$ in the sequence (\ref{syz-sequence}) is strict. Thus (a) implies (b) by Proposition (\ref{Girud1}). Assume that the complex $(*)$  is exact. Then by above $\gr_I(\syz)=\mathrm{Syz}(J^*)$ which yields (c). Finally, we prove that (c) implies (a). The map $\Theta: \mathrm{Syz}(J^*)\lar \gr_I(\syz)$ is inverse of $\varphi$ which  is defined by sending an element $s\in \mathrm{Syz}(J^*) $ to $\psi(a_1,\ldots,a_t)$ where $\varphi\circ\psi((a_1,\ldots,a_t))=s$. Hence $\gr_I(\syz)\simeq \mathrm{Syz}(J^*)$. The latter implies that $J\cap I^n=JI^{n-1}$ for all $n\geq 1$. In fact, let $ b \in J\cap  I^{n}= \ffr(R^t)\cap \FF_n R$, then $ b=\ffr((a_1,\ldots,a_t))$ with $(a_1,\ldots,a_t)$ belonging to some $\FF_mR^t$. If $m\geq n$, we get the assertion. If $m<n$, by using the exactness of   $0\lar \gr_I(\syz)\stackrel{\gr(i)}\lar \gr_I(R^t)$, we get  $$(a_1,\ldots,a_n)\in \ffr^{-1}(\FF_{m+1} R)\cap \FF_m R^t=\FF_m \syz+\FF_{m+1} R^t.$$
Hence $b=\ffr((c_1,\ldots,c_t))$ with $(c_1,\ldots,c_t)\in \FF_{m+1}R^t$. Repeating  this argument finitely many times, finally we get $b\in \ffr(\FF_n R^t)=JI^{n-1}$ which complete the proof. 
\end{proof}
\section{Aluffi torsion-free ideals}
In this section we assume that all rings are Notherian. Let $J\subseteq I$ be ideals in the ring $R$. If $J\subseteq I$ satisfy in one of the equivalent  conditions in  the  Remark (\ref{VV-Aluffi}) or the Theorem (\ref{Torsion-Syzygy}) then the Aluffi algebra is torsion-free. Therefore we have the following definition. 
 
\begin{Definition}\rm
A pair of ideals $J\subseteq I$ in  the ring $R$ is called \textit{Aluffi torsion-free} if $J\cap I^n=JI^{n-1}$ for all $n\geq 1$.  
\end{Definition}
\begin{Example}\rm
There are well-known examples of Aluffi torsion-free ideals. 
\begin{enumerate}
\item If  $ I/J $ in $ R/J$ is of linear type (e.g., if $ I$ is generated by regular or, more generally by a  $d$-sequence modulo $ J$ in the sense of Huneke (\cite{Huneke})) then $J\subseteq I$ is Aluffi torsion-free.  
\item  If  $ J$ is generated by superficial sequence in $ I$ then the pair  $J\subseteq I$ is Aluffi torsion-free \cite[Lemma 8.5.11]{HuSw}. 
\end{enumerate}

\end{Example}
The following result indicate to the behavior of  Aluffi torsion-free property with respect to extension and contraction. In particular, it shows that the Aluffi torsion-free property is local.  
\begin{Proposition}\label{Extention of Rings} Let $J\subseteq I$ be ideals in the ring $R$. The following statements hold: 
\begin{enumerate}
\item [{\rm (a)}] Let $\mathfrak{a}\subseteq J$ be  another ideal.  If $J\subseteq I$ is Aluffi torsion-free then $\overline{J}\subseteq \overline{I}$ is Aluffi torsion-free in $\overline{R}=R/\mathfrak{a}$.
\item [{\rm (b)}] Let  $R\lar S$ be a flat homomorphism of rings. If $J\subseteq I$ is Aluffi torsion-free then $JS\subseteq IS$ is  Aluffi torsion-free in $S$.  
\item [{\rm (c)}] Let $R\lar S$ be a faithfully flat homomorphism of rings.  If the extension of ideals  $J\subseteq I$  in $S$ is Aluffi torsion-free then $J\subseteq I$ is Aluffi torsion-free in $R$. In particular, Assume that  $(R,\fm)$ is local. If  the extension of $J$ and $I$ in the $\fm$-adic completion $\hat{R}$ is Aluffi torsion-free then  so does $J\subseteq I$.  
\item [{\rm (d)}] The ideal  $J\subseteq I$ is Aluffi torsion-free if and only if $JR_{\fm}\subseteq IR_{\fm}$ is Aluffi torsion-free for every maximal ideal $\fm$ of $R$. 
\end{enumerate}
\end{Proposition}
\begin{proof}
We prove (a), (b) and (c) by straightforward computations. We have  
$$\overline{J}\cap \overline{I}^{n}=\overline{J}\cap \overline{I^n}=\overline{J\cap I^n}=\overline{JI^{n-1}}=\overline{J}\ \overline{I}^{n-1},$$
which proves  (a). For (b), we have 
\begin{eqnarray}
\nonumber JS\cap (IS)^n=JS\cap I^nS &=&(J\otimes_R S)\cap (I^n\otimes_R S)\\
\nonumber & =& (J\cap I^n)\otimes_R S=(JI^{n-1})\otimes_R S\\
\nonumber &=& (JI^{n-1})S=(JS) (I^{n-1}S).
\end{eqnarray}
(c). As $S$ is faithfully flat over $R$, $IS\cap R=I$ for all ideals $I$ of $R$. We have 
\[J\cap I^n= (J\cap I^n)S\cap R\subseteq(JS\cap I^nS)\cap R= (JI^{n-1})S\cap R=J I^{n-1}.  \]
The second assertion yields from the fact  that $R\lar \hat{R}$ is  faithfully flat. 
The part (d) follow from part (b) and local-global  property.  
\end{proof}
\begin{Remark}\rm There is a natural question. What is the behavior of Aluffi torsion-free property with respect to operation of ideals?  Here are some easy facts about this question. 
\begin{enumerate}
\item The sum of two Aluffi torsion-free ideals is not Aluffi torsion-free (see Proposition \ref{sum of two ATF}).
\item The product and intersection of two Aluffi torsion-free need not to be Aluffi torsion-free. In $k[x,y,z]$,  consider the ideals $J_1=(xy)$ and $J_2=(yz) \subseteq I=(x,y,z)^2$ which are Aluffi torsion-free, but $J_1J_2=(xy^2z), J_1\cap J_2=(xyz)\subseteq I$ are not Aluffi torsion-free. 
\end{enumerate} 
\end{Remark}
\begin{Proposition}\label{sum of two ATF}
	Let  $ J_{1},J_{2} \subseteq I$ be Aluffi torsion-free ideals in the ring $ R$. Then $J_1+J_2\subseteq I$ is Aluffi torsion-free if and only if $\overline{J_1}\subseteq \overline{I}\subseteq \overline{R}=R/J_2 $ is Aluffi torsion-free.
\end{Proposition}
\begin{proof}
Assume that $J_1+J_2\subseteq I$ is Aluffi torsion-free. For all $n\geq 1$ we have 
\[\overbar{J_1}\cap \overbar{I}^n=\frac{(J_1+J_2)\cap I^n+J_2}{J_2}=\frac{(J_1+J_2)I^{n-1}+J_2}{J_2}=\frac{J_1I^{n-1}+J_2}{J_2}=\overbar{J_1}\ \overbar{I}^{n-1}. \]
For the converse, let $J_1=(f_1,\ldots,f_t)$ and $y\in (J_1+J_2)\cap I^n$, then $y=a+c$ with $a\in J_1$ and $c\in J_2$. If $\overbar{a}=0$ we are done, if not we get $a\in \overbar{J_1}\cap \overbar{I}^n=\overbar{J_1}\overbar{I}^{n-1}$, then $a=\sum_{i=1}^{t}g_if_i+d$ with $g_i\in I^{n-1}$ and $d\in J_2$. It follows that $y=c+d+\sum_{i=1}^{t}g_if_i$, where  $\sum_{i=1}^{s}g_if_i\in J_1\cap I^n$ and $c+d\in J_2\cap I^n$. Hence $(J_1+J_2)\cap I^n\subseteq (J_1\cap I^n)+(J_2\cap I^n)$. Since $J_1,J_2\subseteq I$ are Aluffi torsion-free, one has
\[(J_1+J_2)\cap I^n\subseteq (J_1\cap I^n)+(J_2\cap I^n)=J_1I^{n-1}+J_2I^{n-1}=(J_1+J_2)I^{n-1}. \]

\end{proof}
\begin{Proposition}\label{residual_gens}
	Let $R$ be a  local ring, $J_1,J_2\subset R$ two ideals and $I=J_1+J_2$. Assume that $J_1$ is generated by elements not in $I^2$. The following are equivalent.
	\begin{enumerate}
		\item[{\rm (a)}] The pair $J_1\subset I$ is Aluffi torsion-free.
		\item[{\rm (b)}] $J_1\cap J_2^n=J_1I^{n-1}$ for all $n\geq 1$.
		\item[{\rm (c)}] $\gr_I{R}/J_1^* \simeq \gr_{I/J_1}{(R/J_1)}$.
	\end{enumerate} 
\end{Proposition}
\begin{proof}
	Let $n$ be a positive integer, we have
\begin{eqnarray}
\nonumber J_1\cap I^n   =  J_1\cap (J_1+J_2)^n  & =& J_1\cap (J_1^n, J_1^{n-1}J_2,\ldots, J_1J_2^{n-1},J_2^n)\\
\nonumber & = & J_1\cap (J_1^n, J_1^{n-1}J_2,\ldots, J_1J_2^{n-1} ) + J_1\cap J_2^n\\
\nonumber  & = & J_1(J_1+J_2)^{n-1}+ J_1\cap J_2^n\\
\nonumber   & = & J_1I^{n-1}+ J_1\cap J_2^n,
\end{eqnarray}
	which prove the equivalence of (a) and (b). 
	
	The associated graded ring $\gr_{I/J_1}{(R/J_1)}$ is isomorphic to 
	\[ R/I\oplus I/(I^2,J_1)\oplus I^2/(I^3+ I^2\cap J_1) \oplus\cdots\ . \]
	Since $J_1^*$ is generated by homogeneous elements in degree 1, $\gr_I{R}/J_1^*$ is isomorphic to 
	\[R/I\oplus I/(I^2,J_1)\oplus I^2/(I^3+IJ_1)\oplus\ldots\ . \]
Now	Assume that (c) holds. Hence for every $n\geq 1$
	\begin{equation}\label{gr}
	I^n/(I^{n+1}+J_1I^{n-1})\simeq I^n/(I^{n+1}+I^n\cap J_1).
	\end{equation}
	This shows that $J_1\cap I^n\subseteq I^{n+1}+J_1I^{n-1}$ for every $n\geq 1$. We have 
	\[ J_1\cap I^n\subseteq I^{n+1}\cap J_1+J_1I^{n+1}\subseteq I^{n+2}+J_1I^n+J_1I^{n-1}=I^{n+2}+J_1I^{n-1}.\]
By induction, for all $m\geq 1$
	\[J_1\cap I^{n}\subseteq J_1I^{n-1}+I^{n+m} .\]
Since $R$ is local, we obtain $J_1\cap I^n=J_1I^{n-1}$, which prove (a). The converse is clear by (\ref{gr}).  
\end{proof}

\begin{Theorem}\label{I-adic}
Let $J_1=(f_1,\ldots,f_m)$ and $J_2=(g_1,\ldots,g_m)\subset I$ be ideals in the  ring $R$ and let  $J_1\subset I$ be Aluffi torsion-free. Suppose that $f_i-g_i\in I^2$ for $i=1,\ldots, m$. Let $\mathcal{Z}_1$ and $\mathcal{Z}_2$ stand for  the first syzygies modules of $J_1$ and $J_2$ respectively. Then 
$J_2\subset I$ is Aluffi torsion-free if and only if $\mathcal{Z}_1\cap I^{n}R^m\subseteq \mathcal{Z}_2+ I^{n+1}R^{m}$ for all $n\geq 0$.
\end{Theorem}
\begin{proof}
($\Rightarrow$) If $(a_1,\ldots,a_m)\in \mathcal{Z}_1\cap I^{n}R^m$ then $\sum_{i=1}^m a_if_i=0$ and $a_i\in I^{n}$.  Hence 
$\sum_{i=1}^ma_ig_i=\sum_{i=1}^ma_i(g_i-f_i)\in J_2\cap I^{n+2}=J_2I^{n+1}$ by assumption, then $\sum_{i=1}^mg_i(a_i-b_i)=0$, where
$b_i\in I^{n+1}$ and $(a_1,\ldots,a_n)=((a_1-b_1)+b_1,\ldots,(a_n-b_n)+b_n)\in \mathcal{Z}_2+I^{n+1}R^{m}$.
		
($\Leftarrow$) We only need to show that $J_2\cap I^{n+1}\subseteq J_2I^{n}$ for all $n\geq 0$. Make induction on $n$, the case $n=0$ is trivial.  Pick an element $\sum_{i=1}^m a_ig_i\in J_2\cap  I^{n+1}$. We may assume that $a_i\in I^{n-1}$ for $i=1,\ldots,m$, in fact we have  
$\sum_{i=1}^ma_ig_i\in J_2\cap I^{n+1}\subseteq J_2\cap I^{n}=J_2I^{n-1}$  by induction hypothesis, so that $\sum_{i=1}^ma_ig_i=\sum_{i=1}^ma_i'g_i$ with $a_i'\in I^{n-1}$. 

We may assume that $(a_1,\ldots,a_m)\in \mathcal{Z}_1\cap I^{n-1}R^m$. Namely we have $\sum_{i=1}^{m}a_i(f_i-g_i)\in I^{n-1}I^2=I^{n+1}$ hence $\sum_{i=1}^{m}a_if_i\in J_1\cap I^{n+1}=J_1I^n $ and $\sum_{i=1}^{m}a_if_i=\sum_{i=1}^{m}b_if_i$ with $b_i\in I^n$.  Now we see that $\sum_{i=1}^{m}(a_i-b_i)f_i=0$, so that we may replace $a_i$ with $(a_i-b_i)$ and this gives also $\sum_{i=1}^{m}a_if_i=0$ as required. We have now
\[\sum_{i=1}^{m}a_ig_i\in I^{n+1}\ ; \  (a_1,\ldots,a_m)\in \mathcal{Z}_1\cap I^{n-1}R^m. \]
Now by assumption, we have $(a_1,\ldots,a_m)\in \mathcal{Z}_2+I^nR^m$. Then there exists $b_i\in I^n $ and $(e_1,\ldots,e_m)\in \mathcal{Z}_2$ such that $(a_1,\ldots,a_m)=(e_1+b_1,\ldots,e_m+b_m)$. Then  $\sum_{i=1}^{m}a_ig_i=\sum_{i=1}^{m}b_ig_i$ and replacing the $a_i$ 's with $b_i$'s we may suppose that $a_i\in I^n$. Repeating the first  argument above  with $a_i\in I^n $ we get 
\[\sum_{i=1}^{m}a_ig_i\in I^{n+2}\ ; \  (a_1,\ldots,a_m)\in \mathcal{Z}_1\cap I^{n}R^m. \]
Therefore, we have an element $\sum_{i=1}^{m}a_ig_i$ such that $a_i\in I^n $  and it is clear that such element belong to $J_2I^{n}$. 
\end{proof}
\begin{Corollary}
Let $J_1,J_2\subseteq I$ be ideals in the  ring $R$ such that $J_1\equiv J_2$ modulo $I^2$ and $J_1\subseteq I$ is Aluffi torsion-free.  Then $J_2\subseteq I$ is Aluffi torsion-free if and only if the first syzygy modules of  $J_1,J_2$ have the same form  ideals in $\gr_{I}(R^m)$. 
\end{Corollary}
\begin{proof}
The proof is based on the symmetry of the Theorem (\ref{I-adic}) and  the fact that  the condition $\mathcal{Z}_1\cap I^{n}R^m\subseteq \mathcal{Z}_2+ I^{n+1}R^{m}$  is equivalent with  $(\mathcal{Z}_1)^*\subseteq (\mathcal{Z}_2)^*$ in $\gr_{I}(R^m)=\bigoplus_{n\geq 0}I^nR^m/I^{n+1}R^m$. More precisely, if $a\in \mathcal{Z}_1\cap I^{n}R^m\setminus I^{n+1}R^{m}$ for some $n$ then $a\in \mathcal{Z}_2+ I^{n+1}R^{m} $ hence $a=b+c$ with $b\in \mathcal{Z}_2 $ and $c\in I^{n+1}R^{m} $ thus $b^*=a^*$ which proves that $(\mathcal{Z}_1)^*\subseteq (\mathcal{Z}_2)^*$. Conversely,  if $a\in \mathcal{Z}_1\cap I^{n}R^m $, we choose an element $b\in  \mathcal{Z}_2 $ such that $a^*=b^*$ hence $a-b\in I^{n+1}R^{m}$ and we are done.
\end{proof}	
\begin{Proposition}
	Let $ J_{1}\subseteq J_{2} \subseteq I$ be  ideals in the ring $R$. Assume that   $ J_{1}\cap J_{2}^{n-1}=J_1J_2^{n-1}$  and  $ I^{n}\subseteq J_{2}^{n}+J_{1}$ for all $n\geq 1$. 
	Then $ J_{1}\subseteq I$ is Aluffi torsion-free. 
\end{Proposition}
\begin{proof}
	We show by induction on $n$ that 
	\begin{equation}\label{H1}
	I^n\subseteq J_2^n+\sum_{j=0}^{n-1}I^{j}(IJ_2^{n-j-1}\cap J_1),  \quad \hbox{for all}\  n\geq 1.
	\end{equation}
	The case  $n=1$ is clear  by second assumption. Suppose that (\ref{H1}) holds for some $n$. Multiplying (\ref*{H1}) by $I$ yields $I^{n+1}\subseteq IJ_2^n+\sum_{j=0}^{n-1}I^{j+1}(IJ_2^{n-j-1}\cap J_1)$. Again by second assumption we also have that $I^{n+1}\subseteq J_2^{n+1}+J_1$, so that $I^{n+1}$ contained in 
	\[\left[ IJ_2^n+\sum_{j=0}^{n-1}I^{j+1}(IJ_2^{n-j-1}\cap J_1)  \right]\bigcap \ (J_2^{n+1}+J_1).\]
	Let $a$ be an element of $I^{n+1}$. Write $a=b+c=d+e$ where 
	\[b\in IJ_2^n\ \  , \ \ c\in \sum_{j=0}^{n-1}I^{j+1}(IJ_2^{n-j-1}\cap J_1)\ \ , \ \ d\in J_2^{n+1}\subseteq IJ_2^{n}\ \ ,\ \ e\in J_1. \]
	Then $e-c=d-b$, and so $d-b\in  J_1\cap IJ_2^{n}$. Therefore, $a=d+e=d+c+(e-c)$ is in
	\[J_2^{n+1} + \sum_{j=0}^{n-1}I^{j+1}(IJ_2^{n-j-1}\cap J_1)+ (J_1\cap IJ_2^{n})= J_2^{n+1} + \sum_{j=0}^{n}I^{j}(IJ_2^{n-j}\cap J_1), \] 
	which proves (\ref*{H1}). Now  using (\ref*{H1}) we obtain that 
	\begin{eqnarray}
	\nonumber J_1 \cap I^n &\subseteq &(J_1\cap J_2^n)+(J_1\cap IJ_2^{n-1})+ \sum_{j=1}^{n}I^{j}(IJ_2^{n-j}\cap J_1)\\
	\nonumber  &\subseteq & (J_1\cap J_2^{n-1}) + \sum_{j=1}^{n}I^{j}(IJ_2^{n-j}\cap J_1)\\
	\nonumber &\subseteq & J_1J_2^{n-1}+ \sum_{j=1}^{n}I^{j}(J_2^{n-j}\cap J_1) \subseteq J_1J_2^{n-1}+ \sum_{j=1}^{n}I^{j}(J_1J_2^{n-j})\\
	\nonumber &\subseteq & J_1I^{n-1}+ \sum_{j=1}^{n}I^{j}(J_1I^{n-j})= J_1I^{n-1}
	\end{eqnarray}
	
\end{proof}

\subsection{Strongly Aluffi torsion-free ideals}
Let $R$ be a local ring and $I=(f_1,\ldots,f_t)$ an ideal such that $f_1,\ldots,f_t$ is  a regular sequence. Then  for any  $n\geq 1$ the pair  $J_i=(f_1^n,\ldots,f_i^n)\subset I^n $ is Aluffi torsion-free for $i=1,\ldots,t$ \cite[Example 1.3]{AR}.  We have the following definition. 
\begin{Definition}
The pair $J=(f_1,\ldots,f_t)\subseteq I$ is called strongly Aluffi torsion-free if $J_i=(f_1,\ldots,f_i)\subseteq I$ is Aluffi torsion-free for $i=1,\ldots,t$. 
\end{Definition}
In general, the following example shows that  Aluffi torsion-free property  does not implies strongly Aluffi torsion-free property.
\begin{Example}
Let $J\subseteq k[x,y,z]$ be an ideal of $5$ projective points in general linear position in $\pp_k^2$ which are columns of the matrix 
\[\begin{bmatrix}
1 & 0 & 0 & 1 & -1 \\
0 & 1 & 0 & 1 & 2 \\
0 & 0 & 1 & 1 & 1 \\
\end{bmatrix}.\]
Then $J=(xy+3xz-4yz, zx^2-2yz^2+xz^2, zy^2+6xz^2-7yz^2)$ which is codimension $2$ perfect ideal.  Let $I$ stands for the Jacobian ideal $I=(J,I_2(\Theta))$ where $\Theta$ is the Jacobian matrix of $J$ and $I_2(\Theta)$ is the ideal generated by $2$-minors of $\Theta$. A calculation in ( \cite{singular}) shows that $J\subseteq I$ is Aluffi torsion-free but is not strongly Aluffi torsion-free.  
\end{Example} 
The proposition below gives a criterion for strongly Aluffi torsion-free ideals. 
\begin{Proposition}\label{regular}
Let $J=(f_1,\ldots,f_t)\subseteq I$ be Aluffi torsion-free ideals in the ring $R$. If  $\left(J_{t-s}:_R f_{t-s+1}\right)=J_{t-s}$  for  $1\leq s\leq t-1$,   then  $J\subseteq I$ is strongly Aluffi torsion-free.  
\end{Proposition}
\begin{proof}
It is enough to show that $  J_{t-1}\cap I^{n}= J_{t-1}I^{n-1}$ for any $ n \geq 1$. Let $a$ be an element of $J_{t-1} \cap I^{n} $. Since $ J \subseteq I$ is Aluffi torsion-free and $ J_{t-1}\cap I^{n} \subseteq J\cap  I^{n}$ hence $ a \in JI^{n-1}$. Write $ a=\sum_{i=1}^{t-1}a_{i}f_{i}=\sum_{i=1}^{t}b_{i}f_{i}$ with $a_i\in I^n$ and $ b_{i}\in I^{n-1}$. One has  $$ b_{t}f_{t}=(a_{1}-b_{1})f_{1}+ \ldots + (a_{t-1}-b_{t-1})f_{t-1} \subseteq J_{t-1},$$
Thus $b_r\in (J_{t-1}:f_t)$. We get
$$a=(a_{1}-b_{1})f_{1}+\ldots+(a_{t-1}-b_{t-1})f_{t-1}+b_{t}f_{t} \in J_{t-1}I^{n-1}+f_{t}((J_{t-1}:f_t)\cap I^{n-1}).$$ 
Then  for all $n\geq 1$ we get $$J_{t-1}\cap I^n\subseteq  J_{t-1}I^{n-1}+f_{t}((J_{t-1}:f_t)\cap I^{n-1}). $$
Now by assumption we have 
\[J_{t-1}\cap I^n\subseteq  J_{t-1}I^{n-1}+f_{t}(J_{t-1}\cap I^{n-1}).\]
Making  induction on $n$ we get
\[J_{t-1}\cap I^n\subseteq J_{t-1}I^{n-1}+f_{t}(J_{t-1} I^{n-2})= J_{t-1}I^{n-1},\]
as required. 
\end{proof}
\begin{Remark}\rm
	Let $ J=( f_1,\ldots,f_t) \subseteq R$ be an ideal such that $f_1,\ldots,f_t$ is  a regular sequence. If  $J\subseteq I$ is Aluffi torsion-free then by  Proposition (\ref{regular})  it is strongly Aluffi torsion-free. Also by the proof of Proposition (\ref{regular}), if for all $n\geq 1$ and $1\leq s \leq t-1$ we have $$f_{t-s+1}((J_{t-s}:f_{t-s+1})\cap I^{n})\subseteq J_{t-s}I^{n},$$ then  strongly Aluffi torsion-free property holds. 
\end{Remark}
\begin{Example}
Let $R=k[x_1,\ldots,x_n]$ and $J=(x_ix_j\ : \ 1\leq i<j\leq n)$. By  \cite[Proposition 2.1]{AR}, $J\subseteq I=(J,x_1^{n-1},\ldots, x_n^{n-1})$ is Aluffi torsion-free. Note that the number of generators of $J$ is $t=n(n-1)/2$. We show that $J\subseteq I$ is strongly Aluffi torsion-free. By above remark and symmetry we just  prove that $x_{n-1}x_n((J_{t-1}:_R (x_{n-1}x_n)\cap I^{m} )$ contained in $J_{t-1}I^{m}$ for all $m\geq 1$. An easy calculation show that $Q:=(J_{r-1}:_R (x_{n-1}x_n))=(x_1,\ldots,x_{n-2})$. Setting $\Delta=(\hat{J},x_1^{n-1},\ldots,x_{n-2}^{n-1})$ and $\Gamma=(x_{n-1}x_n,x_{n-1}^{n-1},x_n^{n-1})$, where by $\hat{J}$ we mean $J$ without the generator $x_{n-1}x_n$.  Write $I=(\Gamma, \Delta)$. We have
\begin{eqnarray}
\nonumber x_{n-1}x_n(Q\cap I^{m})&=&x_{n-1}x_n\left[Q\cap (\Gamma^{m}, \Gamma^{m-1}\Delta,\ldots,\Gamma\Delta^{m-1},\Delta^{m})\right]\\
\nonumber &=& x_{n-1}x_n\Delta(I^{m-1})+(x_{n-1}x_n)Q\Gamma^{m}\subseteq J_{t-1}I^m.
\end{eqnarray}
\end{Example}
\begin{Theorem}
Let $J_1,J_2\subseteq I$ be ideals in the ring $R$. Assume that $\overbar{J_2}\subseteq \overbar{I}$ is Aluffi torsion-fee in $\overbar{R}=R/J_1$. If there exists a minimal generators  $f_1,\ldots,f_s$ of $J_1$ such that 
\begin{enumerate}
	\item $J_1=(f_1,\ldots,f_s)\subseteq I$ is strongly Aluffi torsion-free.
	\item $\{\widetilde{f_1},\ldots, \widetilde{f_s}\}$ is a regular sequence in $\widetilde{R}=R/J_2$.
\end{enumerate}  
Then $J_2\subseteq I$ is Aluffi torsion-free.   
\end{Theorem}
\begin{proof}
 We use induction on $ s$. Assume that $s=1$. Since $ \overbar{J_{2}} \subseteq \overbar{I}$ is Aluffi torsion-free then for all $n\geq 1$  we have
 $(J_2\cap I^n)+J_1=(J_2I^{n-1})+J_1$. Intersecting the latter with $J_2\cap I^n$ we get
 \[J_{2}\cap I^{n}=J_{2}I^{n-1}+(J_{1}\cap I^{n}\cap J_{2})\]
 But $ J_{1}\cap I^{n}\cap J_{2}=(f_1)\cap I^{n}\cap J_{2}$. Hence  by (1) we obtain that \[ (f_1)\cap I^{n}\cap J_{2}=(f_1)I^{n-1}\cap J_{2}=(f_1)(I^{n-1}\cap (J_{2}:f_1)).\]
 By (2) $ \tilde{f_1}$ is regular in $ \tilde{R}$, then $ (J_{2}:f_1)=J_{2}$. Hence \[ (f_1)\cap I^{n}\cap J_{2}=f_1(I^{n-1}\cap J_{2}),\]
 and we obtain 
\[J_{2}\cap I^{n}=J_{2}I^{n-1}+ f_1(I^{n-1}\cap J_{2})\]
Now making induction on $n$, we get 
\[J_{2}\cap I^{n}=J_{2}I^{n-1}+ f_1(J_2I^{n-2})\subseteq J_2I^{n-1}. \]
which prove the assertion in this case. Now assume that $ s>1$. Let $ N=(f_{1},...,f_{s-1})$ and  denote by $"\ \hat{}\ "$ reduction modulo $N$. Then in the ring $ \hat{R}$ we have ideals $ \hat{J_1}, \hat{J_2}$ and $\hat{I}$. Furthermore, $\overbar{\hat{J_2}}\subseteq \overbar{\hat{I}}\subseteq \overbar{\hat{R}}$  and  by the minimality of $\{f_1,\ldots,f_s\}$ and Proposition (\ref{sum of two ATF}),   $\hat{J_1}=(\hat{f_s})\subseteq \hat{I}$ is  Aluffi torsion-free. Also $ \hat{f_{s}}$ is regular in $\hat{R}$. Thus by the first step of the induction we get that $\hat{J_2}\subseteq \hat{I}$ is Aluffi torsion-free. Since the ideal $N$ has the same property as the ideal $ J_{1}$ then the inductive assumption complete the proof.
\end{proof}
\section{Application and  examples}
In intersection theory Aluffi algebra is used  for closed embedding of schemes  $Y\hookrightarrow X\hookrightarrow M$ where $M$ is a regular  and $Y$ is the singular subscheme of $X$. In this section we follow this direction. 

Let $R=k[x_0,\ldots,x_n]$ be a polynomial ring over a field $k$ of characteristic zero. Let $J=(f_1,\ldots,f_t)$ be an ideal of height $r$. Denote by $\Theta=(\partial f_{ij}/\partial x_j) $ the Jacobian matrix of $J$ and by $I_r(\Theta)$ the ideals generated by $r$-minors of $\Theta$. The  ideal $I=(J,I_r(\Theta))$ is called the Jacobian ideal of $J$ which describes the singular subscheme of  $\spec{R/J}$.  See (\cite{AR}) and (\cite{ASR}) for examples of Aluffi torsion-free ideals in this situation.  
\begin{Example}\label{affine_curves}\rm
Let $J\subset R=k[x,y,z]$ be the defining ideal of the monomial
space curve with parametric equations $x=u^{n_1},\ y=u^{n_2},\
z=u^{n_3}$, where  $\gcd (n_1,n_2,n_3)=1$. Suppose that
$n_1=2q+1,\ n_2=2q+p+1,\ n_3=2q+2p+1$, for non-negative integers $p,q$.
If $I$ is the Jacobian ideal of $J$ then pair $J\subset I$ is  Aluffi torsion-free.  
\end{Example}
\begin{proof}
Grading $R$ by the exponents of the parameter $u$ in the parametric equations, one knows (\cite {Herzog})  that $J$ is
a perfect codimension $2$ ideal generated by
the homogeneous polynomials $$F_1=x^{c_1}-y^{r_{12}}z^{r_{13}},\
F_2=x^{r_{21}}z^{r_{23}}-y^{c_2},\
F_3=x^{r_{31}}y^{r_{32}}-z^{c_3}$$ where $0<r_{ij}< c_i\ (i=1,2,3,
j\neq i)$. Note the relations
$$c_1=r_{21}+r_{31},\ c_2=r_{12}+r_{32},\  c_3=r_{13}+r_{23}.$$
The Jacobian matrix of $J$ is
$$\Theta=\left(
\begin{array}{ccc}
c_1x^{c_1-1} & -r_{12}y^{r_{12}-1}z^{r_{13}} & -r_{13}y^{r_{12}}z^{r_{13}-1} \\
r_{21}x^{r_{21}-1}z^{r_{23}} & -c_2y^{c_2-1} & r_{23}x^{r_{21}}z^{r_{23}-1} \\
r_{31}x^{r_{31}-1}y^{r_{32}} & r_{32}x^{r_{31}}y^{r_{32}-1} & -c_3z^{c_3-1} \\
\end{array}
\right)
$$
The $2$-minors of $\Theta$ are
\begin{eqnarray}
\nonumber  f_1 &=& -c_1c_2x^{c_1-1}y^{c_2-1}+r_{21}r_{12}x^{r_{21}-1}y^{r_{12}-1}z^{c_3} \\
\nonumber  f_2 &=& c_1r_{23} x^{c_1+r_{21}-1}z^{r_{23}-1}+r_{13}r_{21}x^{r_{21}-1}y^{r_{12}}z^{c_3-1}  \\
\nonumber  f_3 &=& -r_{12}r_{23}x^{r_{21}}y^{r_{12}-1}z^{c_3-1}+c_2r_{13}y^{c_2+r_{12}-1}z^{r_{13}-1} \\
\nonumber  f_4 &=& c_1r_{32}x^{c_1+r_{31}-1}y^{r_{32}-1}+r_{31}r_{12}x^{r_{31}-1}y^{c_2-1}z^{r_{13}} \\
\nonumber  f_5 &=& -c_1c_3x^{c_1-1}z^{c_3-1}+ r_{31}r_{13}x^{r_{31}-1}y^{c_2}z^{r_{13}-1}\\
\nonumber  f_6 &=&  r_{32}r_{13}x^{r_{31}}y^{c_2-1}z^{r_{13}-1}+ c_3r_{12}y^{r_{12}-1}z^{c_3+r_{13}-1} \\
\nonumber  f_7 &=& r_{21}r_{32}x^{c_1-1}y^{r_{32}-1}z^{r_{23}}+c_2r_{31}x^{r_{31}-1}y^{c_2+r_{32}-1} \\
\nonumber  f_8 &=& -r_{21}r_{31}x^{c_1-1}y^{r_{32}}z^{r_{23}-1}-c_3r_{21}x^{r_{21}-1}z^{c_3+r_{23}-1} \\
\nonumber  f_9 &=& -r_{32}r_{23}x^{c_1}y^{r_{32}-1}z^{r_{23}-1}+c_2c_3y^{c_2-1}z^{c_3-1}
\end{eqnarray}
Write $\mathfrak{D}$ for the ideal generated by the following monomials
$$
\begin{array}{ccc}
M_1=x^{r_{21}-1}y^{r_{12}-1}z^{c_3} &\quad  M_2= x^{r_{21}-1}y^{r_{12}}z^{c_3-1}&\quad \ \  M_3=y^{c_2+r_{12}-1}z^{r_{13}-1}\\
M_4=x^{r_{31}-1}y^{c_2-1}z^{r_{13}}&\quad  M_5= x^{r_{31}-1}y^{c_2}z^{r_{13}-1}&\quad\ \   M_6=y^{r_{12}-1}z^{c_3+r_{13}-1}\\
M_7=x^{r_{31}-1}y^{c_2+r_{32}-1} & \quad M_8=x^{r_{21}-1}z^{c_3+r_{23}-1} &M_9=y^{c_2-1}z^{c_3-1}
\end{array}
$$
The following relations come out
\begin{eqnarray}
\nonumber  f_1 &=& -c_1c_2x^{r_{21}-1}y^{r_{12}-1}F_3+(r_{21}r_{12}-c_1c_2)M_1  \\
\nonumber  f_2 &=&  c_1r_{23}x^{r_{21}-1}z^{r_{23}-1}F_1+(r_{31}r_{21}-c_1r_{23})M_2  \\
\nonumber  f_3 &=&  -r_{12}r_{23}y^{r_{12}-1}z^{r_{13}-1}F_2-(r_{12}r_{23}+c_2r_{13})M_3 \\
\nonumber  f_4 &=&  c_1r_{32}x^{r_{31}-1}y^{r_{32}-1}F_1+(r_{31}r_{12}+c_1r_{32})M_4  \\
\nonumber  f_5 &=&  -c_1c_3x^{r_{31}-1}z^{r_{13}-1}F_2+(r_{31}r_{13}-c_1c_3)M_5  \\
\nonumber  f_6 &=&  r_{32}r_{13}y^{r_{12}-1}z^{r_{13}-1}F_3+(r_{21}c_3+r_{32}r_{13})M_6  \\
\nonumber  f_7 &=&  r_{21}r_{32}x^{r_{31}-1}y^{r_{32}-1}F_2+(r_{31}c_2+r_{21}r_{32})M_7   \\
\nonumber  f_8 &=&  -r_{32}r_{23}x^{r_{21}-1}z^{r_{23}-1}F_3-(r_{21}c_3+r_{23}r_{31})M_8 \\
\nonumber  f_9 &=&  r_{32}r_{23}y^{r_{32}-1}z^{r_{23}-1}F_1+(c_2c_3-r_{32}r_{23})M_9,
\end{eqnarray}
By above, $J$ is generated by $$F_1=x^{p+q+1}-yz^q, \ \ F_2=xz-y^2,\ \ F_3=x^{p+q}y-z^{q+1}$$
and the Jacobian ideal is
$$I=(J,\mathfrak{D})=(xz-y^2, x^{p+q+1}, x^{p+q}y,x^{p+q-1}y^2,yz^q,y^2z^{q-1},z^{q+1})$$
Set $\Delta= (x^{p+q+1}, x^{p+q}y,x^{p+q-1}y^2,yz^q,y^2z^{q-1},z^{q+1})$.
By a slight adaptation of Proposition~\ref{residual_gens} it suffices to show that $J\cap \Delta^n\subseteq JI^{n-1}$
for every $n\geq 1$. Since  $J$ is binomial prime ideal  and $\Delta$ is monomial then \cite[Corollary 1.5]{DB} implies that $J\cap \Delta^n$ is generated by binomials $(u-v)h$ where $u-v\in J$ and $h\in R$ is a monomial or $h=(u-v)^c$ for some positive integer $c\geq 1$. As elements  $uh,vh$ belong to $\Delta^n$, an easy calculation show that $(u-v)h\in JI^{n-1}$ as required. 
\end{proof}
\begin{Question}
Let $J\subseteq k[x_1,\ldots,x_n]$ be defining ideal of affine  monomial curve with parametric equation $x_1=u^{n_1},\ldots, x_m=u^{n_m}$. Let $I$ be the Jacobian ideal of $J$. For which types of parametrization, the pair $J\subseteq I$ is Aluffi torsion-free. 
\end{Question}
\begin{Example}
Let $J$ be an ideal in the ring $R=k[x_0,\ldots,x_n]$ with $n\geq 2$ generated  with all square free monomial ideal in degree $r$. Let $I$ stands for the Jacobian ideal of $J$. Then the pair $J\subset I$ is Aluffi torsion-free.
\end{Example}
\begin{proof}
Let $J=(x_{i_1}x_{i_2}\cdots x_{i_r}\ : \ 0\leq  i_1<\ldots<i_r\leq n\  )$. It is well known that $\hht J=r-1$.  The transpose Jacobian matrix of $J$ is 
$$\Theta(J)=\left[\begin{array}{cc|c}
x_{i_1}x_{i_2}\cdots x_{i_{r-1}} & \ 1\leq i_1<\ldots<i_{r-1}\leq n & 0 \\
\hline
& \textbf{*} & \Theta' \\
\end{array}
\right],
$$
where $\Theta'$ is the Jacobian matrix of the ideal $J'$ generated by all square free monomial ideal in degree $r$ in $k[x_1,\ldots,x_{n}]$. By induction on $n$ and elementary columns operation we get that the Jacobian ideal $I$ of $J$ is $I=(J\ ,\ x_i^rx_j^r \ :\ 0\leq i<j\leq n )$. By Proposition ~\ref{residual_gens}, it is enough to show that  for all $t\geq 1$
$$J\cap (x_i^{r}x_j^{r}\ :\ 0\leq i<j\leq n)^t\subseteq JI^{t-1}.$$
The proof of the latter inclusion  is based on the usual algorithmic procedure to find generators of the intersection of monomial ideal.
\end{proof}
\begin{Example}
	\begin{enumerate}
		\item  Let $M$ be a $2\times n$ generic matrix in the polynomial ring $R=k[x_i\ ; \ 1\leq i\leq 2n]$ with $n\geq 3$. Let $J=I_2(M)$. It is well-known that $\hht J=n-1$. Let $I=(J,I_{n-1}(\Theta))$ stands for  the Jacobian ideal of $J$. Since $M$ is the concatenation of $n$ scroll blocks of length $1$, then by \cite[Theorem 2.3]{AR}, $I_{n-1}(\Theta)=(x_i\ ; \ 1\leq i\leq 2n)^{n-1}$. In particular, the pair $J\subseteq I$ is Aluffi torsion-free.  
	\item Let $R=k[x_1,\ldots,x_9]$. Consider the  $3\times3$ generic matrix  $M $ in $R$ 
\[M=	\begin{bmatrix}
		x_1 & x_2 & x_3\\
		x_4 & x_5 & x_6\\
		x_7 & x_8 & x_9
	\end{bmatrix}. 
	\]
The ideal  $J=I_2(M)=(\Delta_1, \Delta_2, \Delta_3)$ has codimension $4$ and the Jacobian matrix of $J$ is of  the form 
	$$
	\Theta(J)=\left[
	\begin{array}{c|c|ccc}
	x_5 &  &  &  &   \\
	x_6 &  &  &  &   \\
	x_8 & \textbf{*} &  & \textbf{*} &   \\
	x_9 &  &  &  &   \\
	\hline
	0 & x_6 &  &  &   \\
	0 & x_9 &  &  \textbf{*} &   \\
		\hline
	0 &0  &  &  &   \\
	0 & 0 &  & \Theta' &   \\
	0 &0  &  &  & 
	\end{array}
	\right],
	$$
where the first block of $\Theta(J)$ is the Jacobian matrix of $\Delta_1=(x_1x_5-x_2x_4,\  x_1x_6-x_3x_4,\  x_1x_8-x_2x_7,\  x_1x_9-x_3x_7)$, the second block is the Jacobian matrix of $\Delta_2=(x_2x_6-x_3x_5, \ x_2x_9-x_3x_8)$ and in the last block $\Theta'$ is the Jacobian matrix of $\Delta_3=I_2(\begin{bmatrix}
	x_4 & x_5 & x_6\\
	x_7 & x_8 & x_9
\end{bmatrix})$. Note that $\hht(\Delta_2, \Delta_3)=3$ and $\hht \Delta_3=2$. We claim that $I_4(\Theta(J))=(x_1,\ldots,x_9)^4$ which proves that the pair $J\subseteq I$ is Aluffi torsion-free. By using the first part of the example with $n=3$, the $3$-minors of second and third blocks of $\Theta(J)$ is generated by $(x_6, x_9)(x_4,x_5,x_6,x_7,x_8,x_9)^2$. One has 
 $$(x_5, x_6, x_8, x_9)(x_6, x_9)(x_4,x_5,x_6,x_7,x_8,x_9)^2\in I_4(\Theta(J)).$$  
 Therefore by  changing the role of $x_1$ and $x_2$ and using above argument the assertion hold.  
\end{enumerate}

\begin{Question}
Let $M$ be a $n\times m$ generic matrix in the polynomial ring $R=k[x_{ij}\ ; \ 1\leq i\leq n \ , \  1\leq i\leq m ]$. Let $J=I_2(M)$ be the ideal generated by $2$-minors of $M$. Let $I$ be the Jacobian ideal of $J$. Is the pair $J\subseteq I$ Aluffi torsion-free?  
\end{Question}

\end{Example}

\begin{Example}
Let $J(f)=(\partial f/\partial x, \partial f/\partial y, \partial f/\partial z)\subseteq R= k[x,y,z]$ denote the gradient ideal of a reduced free divisor line arrangement $X=V(f)$ of degree $3$  in $\pp_k^2$. By \cite[Proposition 3.7]{LisbonProcDekker}, $J(f)$ is codimension $2$ perfect ideal.  Then by Hilbert-Burch theorem $J(f)$ is generated by $2$-minors of the $2\times 3$ matrix of linear forms in $R$. If $Y=V(J(f))\subseteq \pp_k^2$ is non-singular then the Jacobian ideal $I$ of $J(f)$ is $(x,y,z)$-primary. Therefore, by \cite[Corollary 2.7]{AR} the pair $J(f)\subseteq I$ is Aluffi torsion-free. 
\end{Example}
We warm up with a conjecture. 
\begin{Conjecure}
Let $X=V(f)$ be a reduced free divisor of line arrangement in $\pp_k^2$. Let $J(f)$ denote the gradient ideal of $f$ and $I$ stands for the Jacobian ideal of $J(f)$. Then $J(f)\subseteq I$ is Aluffi torsion-free. 
\end{Conjecure}


\begin{thebibliography}{99}

\bibitem{aluffi}
P. Aluffi, Shadows of blow-up algebras, {\it Tohoku Math. J.} {\bf 56} (2004) 593-619.

\bibitem{DB}
D.  Eisenbud and B.  Sturmfels, Binomial ideals, \textit{Duke Math. J.} \textbf{84} (1996) 1--45. 

\bibitem{singular}
W. Decker,  G.-M, Greuel, G. Pfister, H. Sch{\"o}nemann, : 
\newblock {\sc Singular} {4-1-0} -- {A} computer algebra system for polynomial computations.
\newblock {http://www.singular.uni-kl.de} (2016).

\bibitem{fulton}
W. Fulton, {\it Intersection Theory}, Springer-Verlag, Berlin, 1984.


\bibitem{Griaud}
J. Giruad, \'Etude locale des singularit\'es, Curso de 3e cycle, Universit\'e d'Orsay, 1971-1972. 

\bibitem{Herzog}
J. Herzog, Generators and relations of Abelian semigroup and semigroup rings. \textit{Manuscripts Math.} \textbf{3} (1970) 175--193. 

\bibitem{H.H}
 H. Hironaka, Resolution of singularities of an algebraic variety over a field of characteristic zero,  \textit{Ann. of Math.} \textbf{79} (1964)  109-–326.

\bibitem{Huneke}
C. Huneke.  On the symmetric and Rees algebra of an ideal generated by a $d$-sequence, {\it J. Algebra.} {\bf 62} (1980) 268-275.

\bibitem{HuSw}
C. Huneke, I. Swanson.  Integral Closure of Ideals, Rings, and Modules, London Math. Soc. Lecture Note Series \textbf{336} (2006), Cambridge University Press. 

\bibitem{keel}
S. Keel, Intersection theory of linear embeddings, {\it Trans. Amer. Math. Soc.} {\bf 335} (1993) 195-212.

\bibitem{Thesis1}
A. Nasrollah Nejad, The Aluffi Algebra of an Ideal, Ph. D. Theis, Universidade Federal de Pernambuco, Brazil, 2010.


\bibitem{AA}
A. Nasrollah Nejad and A. Simis, The Aluffi algebra, \textit{J. Singularities}. \textbf{3} (2011) 20--47.

\bibitem{ASR}{A. Nasrollah Nejad, A. Simis,  R. Zaare-Nahandi, The Aluffi algebra of the Jacobian of points in projective spaces: Torsion-freeness, \textit{J. Algebra.} \textbf{467} (2016) 268--283.}



\bibitem{AR} A. Nasrollah Nejad and R. Zaare-Nahandi, Aluffi Torsion-free ideals, \textit{J. Algebra.} \textbf{346} (2011) 284--298.

\bibitem{LisbonProcDekker}{A. Simis,  Differential idealizers and algebraic free divisors, {\it in\/} {\sc Commutative Algebra: Geometric, Homological, Combinatorial and Computational Aspects}, {Lecture Notes in Pure and Applied Mathematics} (Eds. A. Corso, P. Gimenez, M. V. Pinto and S. Zarzuela), Chapman \& Hall/CRC, Volume {\bf 244} (2005)
211--226.}

\bibitem{VaVa}
P. Valabrega and G. Valla, Form rings and regular sequences, {\it Nagoya
	Math. J.} {\bf 72} (1978) 91-101.

\bibitem{Wolmbook1}
W. Vasconcelos, {\it Arithmetic of Blowup Algebras}, London Mathematical Society, Lecture Notes Series {\bf 195}, Cambridge University Press, 1994.

\end{thebibliography}
\end{document}